\documentclass[10pt,a4paper]{article}
\usepackage{graphicx}
\usepackage[latin1]{inputenc}
\usepackage{amsmath}
\usepackage{amsfonts}
\usepackage{amssymb}
\usepackage{amsthm}
\usepackage{url}
\author{Alec Edgington \thanks{\tt alec@obtext.com}}
\title{The autoconjugacy of a generalized Collatz map}
\date{June 2012}
\theoremstyle{plain} \newtheorem{conj}{Conjecture}
\theoremstyle{plain} 
\theoremstyle{plain} \newtheorem{thm}{Theorem}

\newcommand{\ee}[1]{\; \textrm{#1}}
\newcommand{\zz}{\mathbb{Z}_2}
\newcommand{\qq}{\mathbb{Q}_2}

\begin{document}

\maketitle

\begin{abstract}
Many of the 2-adic properties of the $3x+1$ map generalize to the analogous $mx+r$ map, where $m$ and $r$ are odd integers. We introduce the corresponding autoconjugacy map, prove some simple properties of it and make some further conjectures in the general setting, including weak versions of the periodicity and divergent trajectories conjectures.
\end{abstract}

\section{Introduction}

The Collatz (or $3x+1$, or Syracuse) problem concerns the behaviour under iteration of the map $T$ defined on the integers by:
\[
T(x) = \Bigg\{ \begin{array}{ll}
x/2 & x \equiv 0 \pmod 2 \\
(3x+1)/2 & x \equiv 1 \pmod 2
\end{array} \ee{.}
\]

This map is easily extended to the 2-adic integers $\zz$. (The definition is the same, parity being determined by the first coefficient in the 2-adic expansion.) This proves to be an interesting context in which to study the problem.

Lagarias\cite{laga} introduced a useful encoding of the behaviour of $T$ under iteration by means of the function $Q_\infty : \zz \to \zz$ defined by $Q_\infty (x) = \sum_{k \in \mathbb{N}} t_k 2^k$ where $t_k \in \{0,1\}$ and $t_k \equiv T^{(k)}(x) \pmod 2$. Bernstein\cite{bern} showed that the inverse function $\Phi = Q_\infty^{-1}$ is given by the formula
\[
\Phi (2^{d_0} + 2^{d_1} + 2^{d_2} + \ldots) = {} - 3^{-1} 2^{d_0} - 3^{-2} 2^{d_1} - 3^{-3} 2^{d_2} - \ldots
\]
where $0 \leq d_0 < d_1 < d_2 < \ldots$.

The functions $\Phi$ and $Q_\infty$ are continuous, measure-preserving bijections on $\zz$.

If we define the \emph{shift map} $\sigma : \zz \to \zz$ by
\[
\sigma (x) = \Bigg\{ \begin{array}{ll}
x/2 & x \equiv 0 \pmod 2 \\
(x-1)/2 & x \equiv 1 \pmod 2
\end{array} \ee{,}
\]
then we have $Q_\infty T = \sigma Q_\infty$ and $T \Phi = \Phi \sigma$. The automorphism group of $\sigma$ is just $\{1, V\}$ where $V(x) = -1 - x$. Monks and Yazinski\cite{moya} consider the function
\[
\Omega = \Phi V Q_\infty : \zz \to \zz \ee{.}
\]
This is an autoconjugacy of $T$: that is, $\Omega T = T \Omega$ and $\Omega^2 = 1$. It is the only nontrivial autoconjugacy of $T$.

Write $\qq$ for the set of rational numbers with odd denominator (when written in lowest terms). This is just the set of 2-adic integers whose 2-adic expansion is eventually periodic.
\begin{description}
\item[Periodicity conjecture] These three statements are equivalent:
\begin{itemize}
\item $Q_\infty (\qq) \subseteq \qq$;
\item $\Omega (\qq) \subseteq \qq$;
\item The $T$-orbit of every rational 2-adic integer is eventually periodic.
\end{itemize}
\item[Divergent trajectories conjecture] These three statements are equivalent:
\begin{itemize}
\item $Q_\infty (\mathbb{Z}^+) \subseteq \qq$;
\item $\Omega(\mathbb{Z}^+) \subseteq \qq$;
\item The $T$-orbit of every positive integer is eventually periodic.
\end{itemize}
\end{description}
Thus the map $\Omega$ is of some interest in relation to the Collatz problem and its variants.

\section{A generalized Collatz map}

The map $T$ is a special case of the more general map
\begin{equation}\label{Tdef}
T_{m,r}(x) = \Bigg\{ \begin{array}{ll}
x/2 & x \equiv 0 \pmod 2 \\
(mx+r)/2 & x \equiv 1 \pmod 2
\end{array} \ee{,}
\end{equation}
in which $m$ and $r$ are odd integers. Similarly to the definition of $Q_\infty$, one may define the corresponding map
\begin{equation}\label{Qdef}
Q_{m,r}(x) = \sum_{k \in \mathbb{N}} t_k 2^k \quad \textrm{where} \; t_k \in \{0, 1\}, t_k \equiv T_{m,r}^{(k)} \pmod 2 \ee{.}
\end{equation}

We clearly have $Q_{m,r} T = \sigma Q_{m,r}$. Note that $T_{3,1} = T$, $Q_{3,1} = Q_\infty$, $T_{1,-1} = \sigma$ and $Q_{1,-1} = 1_{\zz}$. By 'long division', one can show that if $r$ is any odd number then
\[
Q_{1,r}(x) = -x/r \ee{.}
\]

The methods described by Lagarias\cite{laga} suffice to show that $Q_{m,r} : \zz \to \zz$ is continuous, one-to-one and onto, and it preserves the 2-adic norm. Write $\Phi_{m,r} = Q_{m,r}^{-1}$.

The function $\Phi_{m,r}$ satisfies a pair of identities that 'reverse' the definition~\ref{Tdef}:
\begin{equation}\label{phi}
\Phi_{m,r}(x) = \Bigg\{ \begin{array}{l}
\Phi_{m,r}(2x)/2 \\
(m\Phi_{m,r}(1+2x) + r)/2
\end{array} \ee{.}
\end{equation}
The second of these identities may be derived as follows:
\begin{eqnarray*}
\Phi_{m,r}(x) & = & \Phi_{m,r} \sigma (1+2x) \\
&=& \Phi_{m,r} \sigma Q_{m,r} \Phi_{m,r} (1+2x) \\
&=& \Phi_{m,r} Q_{m,r} T_{m,r} \Phi_{m,r} (1+2x) \\
&=& T_{m,r} \Phi_{m,r} (1+2x) \\
&=& (m\Phi_{m,r}(1+2x) + r)/2
\end{eqnarray*}
since $\Phi_{m,r}(1+2x) \equiv 1 \pmod 2$.

This tells us that:
\[
\Phi_{m,r}(1 + 2x) = -\frac{r}{m} + \frac{2}{m} \Phi_{m,r}(x) \ee{,}
\]
which allows us to generalize Bernstein's formula for $\Phi_{3,1}$: indeed,
\begin{equation}\label{phi_expansion}
\Phi_{m,r} (2^{d_0} + 2^{d_1} + 2^{d_2} + \ldots) = -r (m^{-1} 2^{d_0} + m^{-2} 2^{d_1} + m^{-3} 2^{d_2} + \ldots ) \ee{.}
\end{equation}

A notable corollary is that
\[
\Phi_{m,r}(x) = r \Phi_{m,1}(x) \ee{.}
\]
So in some sense, the $mx+r$ problem 'looks like' the $mx+1$ problem.

We define the generalized autoconjugacy by
\[
\Omega_{m,r} = \Phi_{m,r} V Q_{m,r} \ee{.}
\]

Just like the $\Omega = \Omega_{3,1}$ defined above, this satisfies $\Omega_{m,r} T_{m,r} = T_{m,r} \Omega_{m,r}$ and $\Omega_{m,r}^2 = 1$.

For the case $m=1$, we have
\[
\Omega_{1,r}(x) = r - x \ee{.}
\]

\section{Properties of $\Omega_{m,r}$}

The map $\Omega_{m,r}$ is a bijection between the sets of odd and even 2-adic integers. The following corollary of equation~\ref{phi} effectively reduces the study of $\Omega_{m,r}$ to its behaviour on the odd integers.

\begin{thm}\label{omega_formula}
For all $x \in \zz$ and all $n \in \mathbb{N}$,
\[
\frac{r}{m-2} + \Omega_{m,r}(2^n x) = \Big(\frac{2}{m}\Big)^n \Big(\frac{r}{m-2} + \Omega_{m,r}(x)\Big) \ee{.}
\]
\end{thm}
\begin{proof}
We have:
\begin{eqnarray*}
\frac{r}{m-2} + \Omega_{m,r}(2x) &=& \frac{r}{m-2} + \Phi_{m,r}(-1 - Q_{m,r}(2x)) \\
&=& \frac{r}{m-2} + \Phi_{m,r}(-1 - 2 Q_{m,r}(x)) \\
&=& \frac{r}{m-2} + \Phi_{m,r}(1 + 2(-1 - Q_{m,r}(x))) \\
&=& \frac{r}{m-2} - \frac{r}{m} + \frac{2}{m} \Phi_{m,r}(-1 - Q_{m,r}(x)) \\
&=& \frac{2}{m} \Big(\frac{r}{m-2} + \Omega_{m,r}(x)\Big) \ee{.}
\end{eqnarray*}
The proposition follows.
\end{proof}

It is presumably \emph{not} the case in general that $\Omega_{m,r}(x) \in \qq \Rightarrow Q_{m,r}(x) \in \qq$: for a likely counterexample, take $x = \Omega_{5,1}(7)$. However, one can make the following conjecture.
\begin{conj}\label{rational_pairs}
If $\{x, \Omega_{m,r}(x) \} \subseteq \qq$, then $Q_{m,r}(x) \in \qq$.
\end{conj}
This may be regarded as a weak version of the periodicity conjecture. Note that the assertion is symmetrical in $x$ and $\Omega_{m,r}(x)$, because $\Omega_{m,r}(\Omega_{m,r}(x)) = x$ and $Q_{m,r}(\Omega_{m,r}(x)) = -1 - Q_{m,r}(x)$.

\section{The autoconjugacy as a real-valued function}

For $k \in \mathbb{N}$, define the function
\[
Q_{k,m,r} : \zz \to \mathbb{N} \subseteq \zz
\]
by the condition
\[
Q_{k,m,r}(x) \equiv Q_{m,r}(x) \pmod {2^k} \ee{.}
\]
In other words, $Q_{k,m,r}(x)$ is the sum of the first $k$ terms of $Q_{m,r}$ as defined in equation~\ref{Qdef}.

When restricted to $\mathbb{Z}$, $Q_{k,m,r}$ is periodic with period $2^k$, and the induced function $\overline{Q}_{k,m,r} : \mathbb{Z} / 2^k \mathbb{Z} \to \mathbb{Z} / 2^k \mathbb{Z}$ is a permutation, whose order divides $2^k$. The proof of these statements for $m=3, k=1$ is outlined by Lagarias\cite{laga}; the same proof works in the more general case.

Define
\[
\Omega_{k,m,r} = \Phi_{m,r} V Q_{k,m,r} \ee{.}
\]
Since $Q_{k,m,r}(x) \to Q_{m,r}(x)$ in the 2-adic metric as $k \to \infty$, by continuity $\Omega_{k,m,r}(x) \to \Omega_{m,r}(x)$ in the 2-adic metric as $k \to \infty$.

Since $Q_{k,m,r}(x) \in \mathbb{N} \subseteq \qq$, we have $\Omega_{k,m,r}(x) \in \qq \subseteq \mathbb{R}$ for all $x$. Therefore we can ask whether $\Omega_{k,m,r}(x)$ tends to some limit in the \emph{real} metric as $k \to \infty$.

Let
\[
\hat{\Omega}_{m,r}(x) = \lim_{k \to \infty} \Omega_{k,m,r}(x) \in \mathbb{R}
\]
when this limit exists.

The conditions for $\hat{\Omega}_{m,r}(x)$ to exist depend upon $m$ and on the limiting density of even iterates of $T_{m,r}$ starting at $x$.

Let
\[
\nu_{m,r}(x) = \liminf_{k \to \infty} \frac1k \Big\lvert \big\{i < k : T_{m,r}^{(i)}(x) \equiv 0 \pmod 2 \big\} \Big\rvert \ee{.}
\]

\begin{thm}\label{omegahat}
If $\nu_{m,r}(x) > \frac{\log 2}{\log m}$, then $\hat{\Omega}_{m,r}(x)$ exists. If in addition $Q_{m,r}(x) \in \qq$, then $\hat{\Omega}_{m,r}(x) = \Omega_{m,r}(x)$.
\end{thm}
\begin{proof}
Let $0 \leq i_0 < i_1 < i_2 < \ldots$ be the set of $i$ such that $T_{m,r}^{(i)}(x) \equiv 0 \pmod 2$, arranged in increasing order. If $i_j \leq k < i_{j+1}$, then
\begin{eqnarray*}
\Omega_{k,m,r}(x) &=& \Phi_{m,r}(-1 - Q_{k,m,r}(x)) \\
&=& \Phi_{m,r}\Big(\sum_{l<j} 2^{i_l} + \sum_{i \geq k} 2^i\Big) \\
&=& -\frac{r}{m} \Big(\sum_{l<j} m^{-l} 2^{i_l} + \frac{2^k}{m^j} \big(1 - \frac{2}{m}\big)^{-1}\Big)
\end{eqnarray*}
by equation~\ref{phi_expansion}. The condition on the density of the $i_j$ implies that the second term tends to zero. Furthermore, it implies that
\[
\liminf_{j \to \infty} \frac{j}{i_j} > \frac{\log 2}{\log m} \ee{.}
\]
Thus, for some $\delta > 0$ and sufficiently large $j$,
\begin{eqnarray*}
2^{i_j} / m^j &<& 2^{j / (\frac{\log 2}{\log m} + \delta)} / m^j \\
&=& m^{-\epsilon j}
\end{eqnarray*}
where $\epsilon = 1 - (1 + \frac{\log m}{\log 2} \delta)^{-1} > 0$. Thus the first term converges.

If $Q_{m,r}(x) \in \qq$, then the sequence of $i_j$ eventually settles into a periodic pattern (in other words, for some $M$ and $m$, and sufficiently large $j$, $i_{j+m} = i_j + M$); the infinite sum can then be evaluated exactly, and gives the same (rational) result in $\zz$ or $\mathbb{R}$.
\end{proof}

Thus if Conjecture~\ref{rational_pairs} is true, then if we know that $\{x, \Omega_{m,r}(x) \} \subseteq \qq$ and $\nu_{m,r}(x) > \frac{\log 2}{\log m}$, we can conclude that $\hat{\Omega}_{m,r}(x) = \Omega_{m,r}(x)$. The following conjecture is slightly stronger.

\begin{conj}
If $\{x, \Omega_{m,r}(x) \} \subseteq \qq$ and $\hat{\Omega}_{m,r}(x)$ exists then $\hat{\Omega}_{m,r}(x) = \Omega_{m,r}(x)$.
\end{conj}

Note that there are $\{x, \Omega_{5,1}(x) \} \subseteq \qq$ such that $\hat{\Omega}_{5,1}(x)$ doesn't exist: for example, $x = -\frac{14}{17} = \Omega_{5,1}(\frac13) = \Phi_{m,r}(-\frac67)$.

Finally, we conjecture (on the basis of limited numerical evidence) that for integer $x$ and $m \geq 5$, there is a sufficient density of even iterates of $T_{m,r}^{(k)}(x)$ that $\hat{\Omega}_{m,r}(x)$ exists.

\begin{conj}
If $m \geq 5$ and $x \in \mathbb{Z}$, then $\hat{\Omega}_{m,r}(x)$ exists.
\end{conj}
This would follow from Theorem~\ref{omegahat} if we could show that the sequence of iterates $T_{m,r}^{(k)}(x)$ (for integer $x$) does not grow too fast. In this sense, it is a (very) weak version of the divergent trajectories conjecture. It would be interesting to try and better understand how $\nu_{m,r}(x)$ is constrained. Can we even show that $\nu_{m,r}(x) > 0$?

Table~\ref{omega} shows a few values of $\Omega_{5,1}(x)$ and $\hat{\Omega}_{5,1}(x)$ for integer $x$.
\begin{table}
\centering
\begin{tabular}{|c|c|c|}
$x$ & $\Omega_{5,1}(x) \in \zz$ & $\hat{\Omega}_{5,1}(x) \in \mathbb{R}$ \\
\hline
$-9$ & $2^1 (1 + 2^1 + 2^4 + \ldots)$ & $-1.129 \ldots \times 10^4$ \\
$-7$ & $-{160532}/{78125}$ & $-{160532}/{78125}$ \\
$-5$ & $-{3662262}/{1953125}$ & $-{3662262}/{1953125}$ \\
$-3$ & $-{321064}/{78125}$ & $-{321064}/{78125}$ \\
$-1$ & $-2$ & $-2$ \\
$0$ & $-1/3$ & $-1/3$ \\
$1$ & $-{52}/{31}$ & $-{52}/{31}$ \\
$3$ & $-{26}/{31}$ & $-{26}/{31}$ \\
$5$ & $-{464}/{31}$ & $-{464}/{31}$ \\
$7$ & $2^1 (1 + 2^1 + 2^3 + \ldots)$ & $-1.426 \ldots \times 10^2$ \\
$9$ & $2^2 (1 + 2^1 + 2^5 + \ldots)$ & $-1.777 \ldots \times 10^2$
\end{tabular}
\caption{Some values of $\Omega_{5,1}$ and $\hat{\Omega}_{5,1}$}\label{omega}
\end{table}

\end{document}